\documentclass[11pt]{amsart}
\usepackage{amssymb}
\usepackage{epsfig}
\usepackage{mathdots}
 \theoremstyle{plain}
\newtheorem{theorem}{Theorem}
\newtheorem{corollary}{Corollary}

\newtheorem{lemma}{Lemma}
\newtheorem{proposition}{Proposition}
\newtheorem{example}{Example}
\theoremstyle{definition}
\newtheorem{definition}{Definition}
\theoremstyle{remark}

\numberwithin{equation}{section}

\newcommand{\bT}{\begin{theorem}}
\newcommand{\eT}{\end{theorem}}
\newcommand{\bProp}{\begin{proposition}}
\newcommand{\eProp}{\end{proposition}}
\newcommand{\bE}{\begin{example}}
\newcommand{\eE}{\end{example}}
\newcommand{\bL}{\begin{lemma}}
\newcommand{\eL}{\end{lemma}}
\newcommand{\bP}{\begin{proof}}
\newcommand{\eP}{\end{proof}}
\newcommand{\bC}{\begin{corollary}}
\newcommand{\eC}{\end{corollary}}
\newcommand{\bD}{\begin{definition}}
\newcommand{\eD}{\end{definition}}
\newcommand{\be}{\begin{enumerate}}
\newcommand{\ee}{\end{enumerate}}
\newcommand{\beqa}{\begin{eqnarray*}}
\newcommand{\eeqa}{\end{eqnarray*}}
\newcommand{\beqaa}{\begin{eqnarray}}
\newcommand{\eeqaa}{\end{eqnarray}}
\newcommand{\ba}{\begin{array}}
\newcommand{\ea}{\end{array}}

\setbox0=\hbox{$+$}
\newdimen\plusheight
\plusheight=\ht0
\def\+{\;\lower\plusheight\hbox{$+$}\;}

\setbox0=\hbox{$-$}
\newdimen\minusheight
\minusheight=\ht0
\def\-{\;\lower\minusheight\hbox{$-$}\;}

\setbox0=\hbox{$\cdots$}
\newdimen\cdotsheight
\cdotsheight=\plusheight
\def\cds{\lower\cdotsheight\hbox{$\cdots$}}
\begin{document}

\title[A  Reciprocity Relation for WP-Bailey Pairs ]
       {A Reciprocity Relation for WP-Bailey Pairs}

\author{James Mc Laughlin}
\address{Mathematics Department\\
 25 University Avenue\\
West Chester University, West Chester, PA 19383}
\email{jmclaughl@wcupa.edu}

\author{Peter Zimmer}
\address{Mathematics Department\\
25 University Avenue\\
West Chester University, West Chester, PA 19383}
\email{pzimmer@wcupa.edu}


\date{\today}

\begin{abstract}
We derive a new general transformation for WP-Bailey pairs by
considering the a certain limiting case of a WP-Bailey chain
previously found by the authors, and examine several consequences of
this new transformation. These consequences include new summation
formulae involving WP-Bailey pairs.

Other consequences include new proofs of some classical identities
due to Jacobi, Ramanujan and others, and indeed extend these
identities  to identities involving particular specializations of
arbitrary WP-Bailey
 pairs.

\end{abstract}

\maketitle

\section{Introduction}
In the present paper, we derive a new general transformation for
WP-Bailey pairs by considering the a certain limiting case of a
WP-Bailey chain previously found by the authors, and examine several
consequences of this new transformation. These consequences include
new expressions for various theta functions and new summation
formulae involving WP-Bailey pairs.

A \emph{WP-Bailey pair} ( see Andrews \cite{A01}) is a pair of
sequences $(\alpha_{n}(a,k,q)$,\\ $\beta_{n}(a,k,q))$  satisfying $\alpha_{0}(a,k,q)$
$=\beta_{0}(a,k,q)=1$, and for $n>0$,
\begin{align}\label{WPpair}
\beta_{n}(a,k,q) &= \sum_{j=0}^{n}
\frac{(k/a;q)_{n-j}(k;q)_{n+j}}{(q;q)_{n-j}(aq;q)_{n+j}}\alpha_{j}(a,k,q).
\end{align}
If the
context is clear, we occasionally suppress the dependence on some or
all of $a$, $k$ and $q$.
For a WP-Bailey pair  $(\alpha_n(a,k),
\beta_n(a,k))$, define
{\allowdisplaybreaks\begin{multline}\label{6psi6eq1a}
F(a,k,q):=\sum_{n=1}^{\infty}\frac{(1-k
q^{2n})(q;q)_{n-1}\left(\frac{k^2}{a};q\right)_n(qa;q^2)_n}
{(1-k)\left(\frac{qa}{k},kq;q\right)_n\left(\frac{k^2q}{a};q^2\right)_n}
\left(\frac{-qa}{k}\right)^n\beta_n(a,k)\\
-
\sum_{n=1}^{\infty}\frac{(q^2;q^2)_{n-1}
\left(\frac{k^2}{a};q^2\right)_n}
{\left(\frac{q^2a^2}{k^2},q^2a;q^2\right)_n}
\left(\frac{qa}{k}\right)^{2n}\alpha_{2n}(a,k)\\
+\frac{\left(\frac{k^2}{a},\frac{q^3a^2}{k^2},q^3a,q^2;q^2\right)_{\infty}}
{\left(\frac{k^2q}{a},\frac{q^2a^2}{k^2},q^2a,q;q^2\right)_{\infty}}\sum_{n=0}^{\infty}\frac{
\left(\frac{k^2q}{a},q;q^2\right)_n}
{\left(\frac{q^3a^2}{k^2},q^3a;q^2\right)_n}
\left(\frac{qa}{k}\right)^{2n+1}\alpha_{2n+1}(a,k).
\end{multline}
}

The main result of the paper is the following reciprocity result for the function  $F(a,k,q)$.
\begin{theorem}\label{t1}
Let  $a$ and $k$ be non-zero complex numbers and $|q|$ a complex number such that $|q|<\max\{1, |a/k|,|k/a|\}$ and none of the denominators following vanish. Let $(\alpha_n(a,k), \beta_n(a,k))$ be a WP-Bailey pair. Let $F(a,k,q)$  be as at
\eqref{6psi6eq1a}.
Then
{\allowdisplaybreaks
\begin{multline}\label{6psi6eq1aa}
F(a,k,q)-F\left(\frac{1}{a},\frac{1}{k},q\right)\\
=\frac{aq}{k^2}
\displaystyle{\frac{\left(
aq,q/a,k^2/a,q^2a/k^2,k^2/q,q^3/k^2,q^2,q^2;q^2\right)_{\infty}}
   {\left(
aq/k^2,k^2q/a,a,q^2/a,k^2,q^2/k^2,q,q;q^2\right)_{\infty}}}\\
-\frac{a}{k}
\displaystyle{\frac{\left(
k^2/a,qa/k^2,-a,-q/a;q\right)_{\infty}(q^2,q^2;q^2)_{\infty}}
   {\left(
   k^2,q^2/k^2,a^2/k^2,q^2k^2/a^2;q^2\right)_{\infty}}}
 +
\frac{(a^2-k) (a-k^2) }{(1-a) (1-k) (a^2-k^2)}.
\end{multline}}
\end{theorem}

Implications of this result include new representations for various
theta functions. One example is contained in the following identity.
Recall that
\begin{equation}\label{psieq}
\psi(q)=\sum_{n=0}^{\infty}q^{n(n+1)/2}=\frac{(q^2;q^2)_{\infty}}{(q;q^2)_{\infty}}
\end{equation}
is Ramanujan's theta function. Let $\omega:=\exp(2 \pi i/3)$,   let
$\chi_0(n)$ denote the principal character modulo 3, and let
$\psi(q)$ be as defined at \eqref{psieq}. Then
{\allowdisplaybreaks\begin{align*}
9\sum_{n=1}^{\infty}\chi_0(n)\frac{nq^n}{1-q^{2n}}&=\frac{\psi^6(q)}{\psi^2(q^3)}
-\frac{\psi^3(q^{1/2})\psi^3(-q^{1/2})}{\psi(q^{3/2})\psi(-q^{3/2})}\\
&=3(1-\omega^2)\sum_{n=1}^{\infty}\frac{(1-\omega
q^{2n})(q;q)_{n-1}(\omega^2q;q^2)_n(-\omega q)^n}
{(1-q^{3n})(\omega q;q)_n(q;q^2)_n}\notag\\
&\phantom{asd}+3(1-\omega)\sum_{n=1}^{\infty}\frac{(1-\omega^2
q^{2n})(q;q)_{n-1}(\omega q;q^2)_n(-\omega^2 q)^n}
{(1-q^{3n})(\omega^2 q;q)_n(q;q^2)_n}.\notag
\end{align*}}

Another implication is the following transformation for WP-Bailey
pairs. If $(\alpha_n(a,k,q),\beta_n(a,k,q))$ is a WP-Bailey pair,
then {\allowdisplaybreaks
\begin{multline*}
\sum_{n=1}^{\infty}\frac{(1-k
q^{2n})(q;q)_{n-1}\left(\frac{k^2}{q};q\right)_n(q^2;q^2)_n}
{(1-k)\left(\frac{q^2}{k},kq;q\right)_n\left(k^2;q^2\right)_n}
\left(\frac{-q^2}{k}\right)^n\beta_n(q,k,q)
 \\
\phantom{asdfas}-\sum_{n=1}^{\infty}\frac{(q^2;q^2)_{n-1}
\left(\frac{k^2}{q};q^2\right)_nq^{4n}\alpha_{2n}(q,k,q)}
{\left(\frac{q^4}{k^2},q^3;q^2\right)_nk^{2n}}
+\frac{\left(k^2-q\right) \left(k-q^2\right)}{(1-k) (1-q) (k^2-q^2)}
\\
=k\frac{(k^2q,q/k^2,q^2,q^2;q^2)_{\infty}}{(k^2,q^2/k^2,q,q;q^2)_{\infty}}
\left(1+\sum_{n=0}^{\infty}\frac{
\left(\frac{k^2}{q^2},\frac{1}{q};q^2\right)_{n+1}q^{4n+4}\alpha_{2n+1}(q,k,q)}
{\left(\frac{q^3}{k^2},q^2;q^2\right)_{n+1}k^{2n+2}}\right).
\end{multline*}
}

Several other implications are to be found elsewhere in the paper.

\section{background}

The subject of basic hypergeometric series took a leap forward after Andrews development in
\cite{A01} of a  \emph{WP-Bailey Chain}, a mechanism for deriving new WP-bailey
pairs from existing pairs. In \cite{A01}, Andrews in fact
describes two such chains. Warnaar \cite{W03} found four additional
chains and Liu and Ma \cite{LM08} introduced the idea of a
\emph{general WP-Bailey chain}, and discovered one new specific
WP-Bailey chain. In \cite{MZ09}, the authors found three new
WP-Bailey chains.

Each WP-Bailey chain also implies a transformation connecting the terms
in an arbitrary WP-Bailey pair (see \eqref{simsuma3} below for an
example of such a transformation), leading to new transformation formulae  for basic hypergeometric series.

In \cite{McL09a}, the first author
of the present paper went in a somewhat different direction and found two new types of transformations for
WP-Bailey pairs.
\begin{theorem}[McL., \cite{McL09a}]\label{t1older}
If $(\alpha_n(a,k), \beta_n(a,k))$ is a WP-Bailey pair, then subject
to suitable convergence conditions, {\allowdisplaybreaks
\begin{multline}\label{wpeq8}
\sum_{n=1}^{\infty} \frac{(q\sqrt{k},-q\sqrt{k},z;q)_{n}(q;q)_{n-1}}
{\left(\sqrt{k},-\sqrt{k}, q k,\frac{q k}{z};q\right)_{n}}\left( \frac{q a}{ z }\right )^{n} \beta_n(a,k)\\
- \sum_{n=1}^{\infty}
\frac{\left(q\sqrt{\frac{1}{k}},-q\sqrt{\frac{1}{k}},\frac{1}{z};q\right)_{n}(q;q)_{n-1}}
{\left(\sqrt{\frac{1}{k}},-\sqrt{\frac{1}{k}}, \frac{q }{k},\frac{q
z}{k};q\right)_{n}}\left( \frac{q z}{ a }\right )^{n}
\beta_n\left(\frac{1}{a},\frac{1}{k}\right) -\\
 \sum_{n=1}^{\infty}\frac{(z;q)_{n}(q;q)_{n-1}}{\left(q a ,\frac{q
a}{z};q\right)_n}\left (\frac{q a}{z}\right)^{n}\alpha_n(a,k)
+
\sum_{n=1}^{\infty}\frac{\left(\frac{1}{z};q\right)_{n}(q;q)_{n-1}}{\left(\frac{q}{a}
,\frac{q z}{a};q\right)_n}\left (\frac{q
z}{a}\right)^{n}\alpha_n\left(\frac{1}{a},\frac{1}{k}\right)\\=
\frac{(a-k)\left(1-\frac{1}{z}\right)\left(1-\frac{ak}{z}\right)}
{(1-a)(1-k)\left(1-\frac{a}{z}\right)\left(1-\frac{k}{z}\right)}
+\frac{z}{k}\frac{\left(z,\frac{q}{z},\frac{k}{a},
\frac{qa}{k},\frac{ak}{z},\frac{qz}{ak},q,q;q\right)_{\infty}}
{\left(\frac{z}{k},\frac{qk}{z},\frac{z}{a},
\frac{qa}{z},a,\frac{q}{a},k,\frac{q}{k};q\right)_{\infty}}.
\end{multline}
}
\end{theorem}

\begin{theorem}[McL., \cite{McL09a}]\label{c3}
If $(\alpha_n(a,k,q), \beta_n(a,k,q))$ is a WP-Bailey pair, then
subject to suitable convergence conditions, {\allowdisplaybreaks
\begin{multline}\label{wpeq2n}
\sum_{n=1}^{\infty} \frac{(1-k q^{2n})(z;q)_{n}(q;q)_{n-1}}
{(1-k)( q k,q k/z;q)_{n}}\left( \frac{q a}{ z }\right )^{n} \beta_n(a,k,q)\\
+  \sum_{n=1}^{\infty} \frac{(1+k q^{2n})(z;q)_{n}(q;q)_{n-1}}
{(1+k)( -q k,-q k/z;q)_{n}}\left( \frac{-q a}{ z }\right )^{n}
\beta_n(-a,-k,q)\\-2 \sum_{n=1}^{\infty} \frac{(1-k^2
q^{4n})(z^2;q^2)_{n}(q^2;q^2)_{n-1}} {(1-k^2)( q^2 k^2,q^2
k^2/z^2;q^2)_{n}}\left( \frac{q^2 a^2}{ z^2 }\right )^{n}
\beta_n(a^2,k^2,q^2)\\=
\sum_{n=1}^{\infty}\frac{(z;q)_{n}(q;q)_{n-1}}{(q a ,q
a/z;q)_n}\left (\frac{q a}{z}\right)^{n}\alpha_n(a,k,q)\\ +
\sum_{n=1}^{\infty}\frac{(z;q)_{n}(q;q)_{n-1}}{(-q a ,-q
a/z;q)_n}\left (\frac{-q a}{z}\right)^{n}\alpha_n(-a,-k,q)\\-2
\sum_{n=1}^{\infty}\frac{(z^2;q^2)_{n}(q^2;q^2)_{n-1}}{(q^2 a^2 ,q^2
a^2/z^2;q^2)_n}\left (\frac{q^2
a^2}{z^2}\right)^{n}\alpha_n(a^2,k^2,q^2).
\end{multline}
}
\end{theorem}

Some similar results were obtained in two other papers,
\cite{McL09b, McL09c}, and various consequences of the
transformations found were also examined. The results in the present paper may be viewed
as deriving from a  continuation of the investigations in the papers alluded to above.

\section{Proof of the main identity}

Before coming to the proofs, we recall an identity derived by the present authors in \cite{McLZ08}.
\begin{theorem}\text{$($}\cite[Mc Laughlin, Zimmer]{McLZ08}\text{$)$}
If $(\alpha_n(a,k), \beta_n(a,k)$ is a WP-Bailey pair, then
{\allowdisplaybreaks
\begin{multline}\label{simsuma3}
\frac{ (q a b/k, k q/b;q)_{\infty}}
 { (k q,q a/k;q)_{\infty}}\\
 \times  \sum_{n=0}^{\infty}
\frac{(q\sqrt{k}, -q\sqrt{k},k^2/ab,b,\sqrt{q a},- \sqrt{q
a};q)_n}{(\sqrt{k},-\sqrt{k}, q a b/k,k q/b, k \sqrt{q/a}, -
k\sqrt{q/a};q)_n}\left ( \frac{-q
a}{k}\right)^n \beta_n \\
= \frac{\left(\frac{q k^2}{ab}, b q,\frac{ q^2 a^2 b}{k^2}, \frac{q^2
a}{b};q^2\right)_{\infty}}{\left(q,\frac{k^2 q}{a}, q^2 a, \frac{q^2
a^2}{k^2};q^2\right)_{\infty}}
\sum_{n=0}^{\infty}\frac{\left(\frac{k^2}{a b},b;q^2
\right)_{n}}{\left(\frac{q^2 a^2 b}{k^2}, \frac{q^2 a}{b};
q^2\right)_{n}} \left(\frac{-q a}{k}\right)^{2n} \alpha_{2n}
\phantom{asdsfasdf}\\
+ \frac{\left(\frac{k^2}{ab}, b,\frac{ q^3 a^2 b}{k^2}, \frac{q^3
a}{b};q^2\right)_{\infty}}{\left(q,\frac{k^2 q}{a}, q^2 a, \frac{q^2
a^2}{k^2};q^2\right)_{\infty}}
\sum_{n=0}^{\infty}\frac{\left(\frac{k^2 q}{a b},b q;q^2
\right)_{n}}{\left(\frac{q^3 a^2 b}{k^2}, \frac{q^3 a}{b};
q^2\right)_{n}} \left(\frac{-q a}{k}\right)^{2n+1} \alpha_{2n+1}.
\end{multline}
}
\end{theorem}

We also recall a result from \cite{McL09a}, namely that if $f(a,k,z,q)$ is as defined by
{\allowdisplaybreaks\begin{align}\label{fakzqeq}
f(a,&k,z,q)=
\sum_{n=1}^{\infty}
\frac{(q\sqrt{k},-q\sqrt{k},k,z,k/a;q)_{n}}
{(\sqrt{k},-\sqrt{k}, q k,q k/z,q a;q)_{n}(1-q^n)}\left( \frac{q a}{ z }\right )^{n}\\
&=-\sum_{n=1}^{\infty}
\frac{(q\sqrt{a},-q\sqrt{a},a,z,a/k;q)_{n}}
{(\sqrt{a},-\sqrt{a}, q a,q a/z,q k;q)_{n}(1-q^n)}\left( \frac{q k}{ z }\right )^{n}\notag\\
&=\sum_{n=1}^{\infty}\frac{k q^n}{1-kq^n}+ \sum_{n=1}^{\infty}\frac{ q^n a/z}{1-q^n a/z}-
\sum_{n=1}^{\infty}\frac{a q^n}{1-aq^n}- \sum_{n=1}^{\infty}\frac{ q^n k/z}{1-q^n k/z},\notag
\end{align}
}
 then
\begin{multline}\label{wpeq7}
f(a,k,z,q)-f\left(\frac{1}{a},\frac{1}{k},\frac{1}{z},q\right)=
\frac{(a-k)(1-1/z)(1-ak/z)}{(1-a)(1-k)(1-a/z)(1-k/z)}\\
+\frac{z}{k}\frac{(z,q/z,k/a,qa/k,ak/z,qz/ak,q,q;q)_{\infty}}
{(z/k,qk/z,z/a,qa/z,a,q/a,k,q/k;q)_{\infty}}.
\end{multline}
We remark that the identity \eqref{wpeq7} was also proved by the authors in  \cite{ALL01}, for the final form of $f(a,k,z,q)$ above. We note two special cases for the proof below.
Firstly,
\begin{equation}\label{fakzeqa}
f(a/k,k,-1,q)=2\sum_{n=1}^{\infty}\frac{k q^n}{1-k^2q^{2n}}
-2\sum_{n=1}^{\infty}\frac{a q^n/k}{1-a^2q^{2n}/k^2}
\end{equation}
\begin{multline}\label{wpeq7a}
f(a/k,k,-1,q)-f\left(k/a,1/k,-1,q\right)\\=
2\frac{(a/k-k)(1+a)}{(1-a^2/k^2)(1-k^2)}
-2\frac{a}{k}\frac{(k^2/a,qa/k^2,-a,-q/a;q)_{\infty}(q^2,q^2;q^2)_{\infty}}
{(k^2,q^2/k^2,a^2/k^2,q^2 k^2/a^2;q^2)_{\infty}}.
\end{multline}
Secondly,
\begin{multline}\label{fakzeqb}
f\left(a,k^2,aq,q^2\right)\\=
\sum_{n=1}^{\infty}\frac{k^2 q^{2n}}{1-k^2q^{2n}}+ \sum_{n=1}^{\infty}\frac{ q^{2n}/q }{1-q^{2n}/q}-
\sum_{n=1}^{\infty}\frac{a q^{2n}}{1-aq^{2n}}- \sum_{n=1}^{\infty}\frac{ \frac{k^2}{aq}q^{2n}}{1-\frac{k^2}{aq}q^{2n}},
\end{multline}
{\allowdisplaybreaks\begin{multline}\label{wpeq7b}
f\left(a,k^2,aq,q^2\right)
-f\left(\frac{1}{a},\frac{1}{k^2},\frac{1}{aq},q^2\right)\\=
\frac{(a-k^2)\left(1-\frac{1}{aq}\right)
\left(1-\frac{k^2}{q}\right)}
{(1-a)(1-k^2)\left(1-\frac{1}{q}\right)
\left(1-\frac{k^2}{aq}\right)}\\
+\frac{aq}{k^2}
\frac{(aq,q/a,k^2/a,q^2 a/k^2,k^2/q,q^3/k^2,q^2,q^2;q^2)_{\infty}}
{(aq/k^2, k^2q/a,q,q,a,q^2/a,k^2,q^2/k^2;q^2)_{\infty}}.
\end{multline}}

\begin{proof}[Proof of Theorem \ref{t1}]
Rewrite \eqref{simsuma3} as
{\allowdisplaybreaks
\begin{multline}\label{simsuma33}
\frac{ (q a b/k, k q/b;q)_{\infty}}
 { (k q,q a/k;q)_{\infty}}\\
 \times  \sum_{n=1}^{\infty}
\frac{(1-k q^{2n})(k^2/ab,b;q)_n (qa;q^2)_n}
{(1-k)(q a b/k,k q/b;q)_n(k^2q/a;q^2)_n}\left ( \frac{-q
a}{k}\right)^n \beta_n(a,k) \\
-\frac{\left(\frac{q k^2}{ab}, b q,\frac{ q^2 a^2 b}{k^2}, \frac{q^2
a}{b};q^2\right)_{\infty}}{\left(q,\frac{k^2 q}{a}, q^2 a, \frac{q^2
a^2}{k^2};q^2\right)_{\infty}}
\sum_{n=1}^{\infty}\frac{\left(\frac{k^2}{a b},b;q^2
\right)_{n}}{\left(\frac{q^2 a^2 b}{k^2}, \frac{q^2 a}{b};
q^2\right)_{n}} \left(\frac{q a}{k}\right)^{2n} \alpha_{2n}(a,k)
\phantom{asdsfasdf}\\
+ \frac{\left(\frac{k^2}{ab}, b,\frac{ q^3 a^2 b}{k^2}, \frac{q^3
a}{b};q^2\right)_{\infty}}{\left(q,\frac{k^2 q}{a}, q^2 a, \frac{q^2
a^2}{k^2};q^2\right)_{\infty}}
\sum_{n=0}^{\infty}\frac{\left(\frac{k^2 q}{a b},b q;q^2
\right)_{n}}{\left(\frac{q^3 a^2 b}{k^2}, \frac{q^3 a}{b};
q^2\right)_{n}} \left(\frac{q a}{k}\right)^{2n+1} \alpha_{2n+1}(a,k)\\
=\frac{\left(\frac{q k^2}{ab}, b q,\frac{ q^2 a^2 b}{k^2}, \frac{q^2
a}{b};q^2\right)_{\infty}}{\left(q,\frac{k^2 q}{a}, q^2 a, \frac{q^2
a^2}{k^2};q^2\right)_{\infty}}-\frac{ (q a b/k, k q/b;q)_{\infty}}
 { (k q,q a/k;q)_{\infty}}.
\end{multline}
}
If we   divide through by $1-b$  and then let $b \to 1$ on the left side of \eqref{simsuma33}, then  the result is $F(a,k,q)$. If we define
\[
H(b):=\frac{ (q a b/k, k q/b;q)_{\infty}}
 { (k q,q a/k;q)_{\infty}}\frac{\left(q,\frac{k^2 q}{a}, q^2 a, \frac{q^2
a^2}{k^2};q^2\right)_{\infty}}{\left(\frac{q k^2}{ab}, b q,\frac{ q^2 a^2 b}{k^2}, \frac{q^2
a}{b};q^2\right)_{\infty}}
\]
then we see that dividing the right side of \eqref{simsuma33} by $1-b$  gives
\[
\frac{\left(\frac{q k^2}{ab}, b q,\frac{ q^2 a^2 b}{k^2}, \frac{q^2
a}{b};q^2\right)_{\infty}}{\left(q,\frac{k^2 q}{a}, q^2 a, \frac{q^2
a^2}{k^2};q^2\right)_{\infty}}\frac{1-H(b)}{1-b},
\]
so that the result of letting $b\to 1$ is
{\allowdisplaybreaks
\begin{align}\label{G'eq}
H'(1)&=\sum_{n=1}^{\infty}\frac{k q^n}{1-kq^{n}}
-\sum_{n=1}^{\infty}\frac{a q^n/k}{1-aq^{n}/k}\\
&\phantom{a}+
\sum_{n=1}^{\infty}\frac{a^2 q^{2n}/k^2}{1-a^2q^{2n}/k^2}
+ \sum_{n=1}^{\infty}\frac{ q^{2n}/q }{1-q^{2n}/q}
-\sum_{n=1}^{\infty}\frac{a q^{2n}}{1-aq^{2n}}
- \sum_{n=1}^{\infty}\frac{ \frac{k^2}{aq}q^{2n}}{1-\frac{k^2}{aq}q^{2n}} \notag\\
&=\sum_{n=1}^{\infty}\frac{k q^n}{1-k^2q^{2n}}
-\sum_{n=1}^{\infty}\frac{a q^n/k}{1-a^2q^{2n}/k^2}\notag\\
&\phantom{asdf}+
\sum_{n=1}^{\infty}\frac{k^2 q^{2n}}{1-k^2q^{2n}}+ \sum_{n=1}^{\infty}\frac{ q^{2n}/q }{1-q^{2n}/q}-
\sum_{n=1}^{\infty}\frac{a q^{2n}}{1-aq^{2n}}- \sum_{n=1}^{\infty}\frac{ \frac{k^2}{aq}q^{2n}}{1-\frac{k^2}{aq}q^{2n}}\notag\\
&=\frac{1}{2}f(a/k,k,-1,q) + f\left(a,k^2,aq,q^2\right).\notag
\end{align}}
Here the first equality is by logarithmic differentiation (noting that $H(1)=1$), the second is by simple combination/separation of some of the series, and the final equality follows from \eqref{fakzeqa} and \eqref{fakzeqb}. Thus we have that
\begin{equation}\label{Ffeq}
F(a,k,q)=\frac{1}{2}f\left(\frac{a}{k},k,-1,q\right) + f\left(a,k^2,aq,q^2\right).
\end{equation}
Upon replacing $a$ with $1/a$ and $k$ with $1/k$ in $F(a,k,q)$ and the Lambert series above, we get
\[
F\left(\frac{1}{a},\frac{1}{k},q\right)=\frac{1}{2}f\left(\frac{k}{a},\frac{1}{k},-1,q\right) + f\left(\frac{1}{a},\frac{1}{k^2},\frac{1}{aq},q^2\right)
+\frac{q}{1-q}-\frac{\frac{k^2q}{a}}{1-\frac{k^2q}{a}}.
\]
Thus,
\begin{multline*}
F(a,k,q)-F\left(\frac{1}{a},\frac{1}{k},q\right)=\frac{1}{2}f\left(\frac{a}{k},k,-1,q\right)
-\frac{1}{2}f\left(\frac{k}{a},\frac{1}{k},-1,q\right)\\
+f\left(a,k^2,aq,q^2\right)-f\left(\frac{1}{a},\frac{1}{k^2},\frac{1}{aq},q^2\right)
-\frac{q}{1-q}+\frac{\frac{k^2q}{a}}{1-\frac{k^2q}{a}},
\end{multline*}
and \eqref{6psi6eq1a} follows from \eqref{wpeq7a} and \eqref{wpeq7b}, upon noting that
{\allowdisplaybreaks
\begin{multline*}
\frac{(a/k-k)(1+a)}{(1-a^2/k^2)(1-k^2)}
+
\frac{(a-k^2)\left(1-\frac{1}{aq}\right)
\left(1-\frac{k^2}{q}\right)}
{(1-a)(1-k^2)\left(1-\frac{1}{q}\right)
\left(1-\frac{k^2}{aq}\right)}\\
-\frac{q}{1-q}+\frac{\frac{k^2q}{a}}{1-\frac{k^2q}{a}}=
\frac{(a^2-k) (a-k^2) }{(1-a) (1-k) (a^2-k^2)}.
\end{multline*}
}
\end{proof}

One easy implication is the following  summation formula.

\begin{corollary}\label{c1}
Let  $a$ and $k$ be non-zero complex numbers and $|q|$ a complex number such that $|q|<\max\{1, |a/k|,|k/a|\}$ and suppose none of the denominators following vanish. Then
{\allowdisplaybreaks
\begin{multline}\label{c1eq}
\sum_{n=1}^{\infty}\frac{(1-k q^{2n})(q;q)_{n-1}\left(k^2/a,k,k/a;q\right)_n(qa;q^2)_n}
{(1-k)\left(qa/k,kq, aq,q;q\right)_n\left(k^2q/a;q^2\right)_n}
\left(\frac{-qa}{k}\right)^n\\
-
\sum_{n=1}^{\infty}\frac{(1- q^{2n}/k)(q;q)_{n-1}\left(a/k^2,1/k,a/k;q\right)_n\left(q/a;q^2\right)_n}
{(1-1/k)\left(qk/a,q/k, q/a,q;q\right)_n\left(aq/k^2;q^2\right)_n}
\left(\frac{-qk}{a}\right)^n\\
=\frac{aq}{k^2}
\displaystyle{\frac{\left(
aq,q/a,k^2/a,q^2a/k^2,k^2/q,q^3/k^2,q^2,q^2;q^2\right)_{\infty}}
   {\left(
aq/k^2,k^2q/a,a,q^2/a,k^2,q^2/k^2,q,q;q^2\right)_{\infty}}}\\
-\frac{a}{k}
\displaystyle{\frac{\left(
k^2/a,qa/k^2,-a,-q/a;q\right)_{\infty}(q^2,q^2;q^2)_{\infty}}
   {\left(
   k^2,q^2/k^2,a^2/k^2,q^2k^2/a^2;q^2\right)_{\infty}}}
 +
\frac{(a^2-k) (a-k^2) }{(1-a) (1-k) (a^2-k^2)}.
\end{multline}
}
\end{corollary}

\begin{proof}
Insert the``trivial" Bailey pair
\begin{align}\label{tp}
\alpha_{n}(a,q)&=\begin{cases} 1&n=0, \\
0, &n>0,
\end{cases} \\
\beta_n(a,q)&=\frac{(k,k/a;q)_n}{(aq,q;q)_n}.\notag
\end{align}
into \eqref{6psi6eq1aa}.
\end{proof}

Inserting the unit WP-Bailey pair,
\begin{align}\label{up}
\alpha_{n}(a,k)&=\frac{(q \sqrt{a}, -q
\sqrt{a},a,a/k;q)_n}{(\sqrt{a},-\sqrt{a},q,kq;q)_n}\left(\frac{k}{a}\right)^n,\\
\beta_n(a,k)&=\begin{cases} 1&n=0, \notag\\
0, &n>1,
\end{cases}
\end{align}
likewise leads to a four-term summation formula, with the same right side as \eqref{c1eq}.

\section{New $\theta$-Function Identities and new
transformations for basic hypergeometric series.}

We consider some other implications of \eqref{6psi6eq1a} and
\eqref{6psi6eq1aa}.

\begin{corollary}\label{cRS}
Let $|q|<1$ and $(\alpha_n(a,k,q),\beta_n(a,k,q))$ be a WP-Bailey
pair. Then {\allowdisplaybreaks
\begin{multline}\label{lambertrameq}
\sum_{n=1}^{\infty}\frac{(1-k
q^{2n})(q;q)_{n-1}\left(\frac{k^2}{q};q\right)_n(q^2;q^2)_n}
{(1-k)\left(\frac{q^2}{k},kq;q\right)_n\left(k^2;q^2\right)_n}
\left(\frac{-q^2}{k}\right)^n\beta_n(q,k,q)
 \\
\phantom{asdfas}-\sum_{n=1}^{\infty}\frac{(q^2;q^2)_{n-1}
\left(\frac{k^2}{q};q^2\right)_nq^{4n}\alpha_{2n}(q,k,q)}
{\left(\frac{q^4}{k^2},q^3;q^2\right)_nk^{2n}}
+\frac{\left(k^2-q\right) \left(k-q^2\right)}{(1-k) (1-q) (k^2-q^2)}
\\
=k\frac{(k^2q,q/k^2,q^2,q^2;q^2)_{\infty}}{(k^2,q^2/k^2,q,q;q^2)_{\infty}}
\left(1+\sum_{n=0}^{\infty}\frac{
\left(\frac{k^2}{q^2},\frac{1}{q};q^2\right)_{n+1}q^{4n+4}\alpha_{2n+1}(q,k,q)}
{\left(\frac{q^3}{k^2},q^2;q^2\right)_{n+1}k^{2n+2}}\right).
\end{multline}
}
\end{corollary}

\begin{proof}
From \eqref{G'eq} and \eqref{Ffeq},
{\allowdisplaybreaks\begin{align}\label{lambeq1}
F(q,k,q)&=\sum_{n=1}^{\infty}\frac{k q^n}{1-k^2q^{2n}}
-\sum_{n=1}^{\infty}\frac{q^{n+1}/k}{1-q^{2n+2}/k^2}
+\frac{q}{1-q}-\frac{k^2}{1-k^2}\notag\\
&=\sum_{n=1}^{\infty}\frac{k q^n}{1-k^2q^{2n}}
-\sum_{n=1}^{\infty}\frac{q^{n}/qk}{1-q^{2n}/q^2k^2}\notag\\
&\phantom{sdasdasd}+
\frac{1/k}{1-1/k^2}+
\frac{q/k}{1-q^2/k^2}
+\frac{q}{1-q}-\frac{k^2}{1-k^2}\notag\\
&=\sum_{n=1}^{\infty}\frac{k q^n}{1-k^2q^{2n}}
-\sum_{n=1}^{\infty}\frac{q^{n}/qk}{1-q^{2n}/q^2k^2}
+\frac{\left(k^2-q\right) \left(q^2-k\right)}{(1-k) (1-q)
(k^2-q^2)}\notag\\
&=k\frac{(k^2q,q/k^2,q^2,q^2;q^2)_{\infty}}{(k^2,q^2/k^2,q,q;q^2)_{\infty}}
+ \frac{\left(k^2-q\right) \left(q^2-k\right)}{(1-k) (1-q)
(k^2-q^2)},
\end{align}}
where an identity of Ramanujan (\cite[Chapter. 17, page 116, Equation
(8.5)]{B91}) is used to combine the two Lambert series  to give
 the infinite product.

Upon using \eqref{6psi6eq1a} to substitute for $F(q,k,q)$ above, we
get the identity {\allowdisplaybreaks
\begin{multline}
\sum_{n=1}^{\infty}\frac{(1-k
q^{2n})(q;q)_{n-1}\left(\frac{k^2}{q};q\right)_n(q^2;q^2)_n}
{(1-k)\left(\frac{q^2}{k},kq;q\right)_n\left(k^2;q^2\right)_n}
\left(\frac{-q^2}{k}\right)^n\beta_n(q,k)\\
- \sum_{n=1}^{\infty}\frac{(q^2;q^2)_{n-1}
\left(\frac{k^2}{q};q^2\right)_n}
{\left(\frac{q^4}{k^2},q^3;q^2\right)_n}
\left(\frac{q^2}{k}\right)^{2n}\alpha_{2n}(q,k)\\
+\frac{\left(\frac{k^2}{q},\frac{q^5}{k^2},q^4,q^2;q^2\right)_{\infty}}
{\left(k^2,\frac{q^4}{k^2},q^3,q;q^2\right)_{\infty}}\sum_{n=0}^{\infty}\frac{
\left(k^2,q;q^2\right)_n} {\left(\frac{q^5}{k^2},q^4;q^2\right)_n}
\left(\frac{q^2}{k}\right)^{2n+1}\alpha_{2n+1}(q,k)\\
=k\frac{(k^2q,q/k^2,q^2,q^2;q^2)_{\infty}}{(k^2,q^2/k^2,q,q;q^2)_{\infty}}
+ \frac{\left(k^2-q\right) \left(q^2-k\right)}{(1-k) (1-q)
(k^2-q^2)}.
\end{multline}
}

 This last identity gives \eqref{lambertrameq}, after some simple
manipulations.
\end{proof}

Remark: We note in passing that if $R(q)$ denotes the left side of
\eqref{lambertrameq}, and $S(q)$ denotes the infinite series
following the infinite product on the right side of
\eqref{lambertrameq}, then
\[
\frac{(1-k^2)R(q)}{kS(q)}
\]
is invariant under the transformation $k \to 1/k$.

Let
\begin{equation}
M(k,q):=\frac{(k^2q,q/k^2,q^2,q^2;q^2)_{\infty}}{(k^2,q^2/k^2,q,q;q^2)_{\infty}},
\end{equation}
the infinite product at \eqref{lambertrameq} above. We also note
that
 many of the theta
functions investigated by Ramanujan and others are expressible in
terms of $M(k,q)$, so that inserting the WP-Bailey pair
\begin{align}\label{pr4}
\alpha_{n}(a,k)&=\begin{cases}0,&\text{if $n$ is
odd},\\
\frac{\left(q^2 \sqrt{a}, \,-q^2
\sqrt{a},\,a,\,a^2/k^2;\,q^2\right)_{n/2}}
{\left(\sqrt{a},\,-\sqrt{a},\,q^2,\,q^2k^2/a;\,q^2\right)_{n/2}}
\left(\frac{k}{a}\right)^n,&\text{if $n$ is even},
\end{cases}\\
\beta_n(a,k)&=\frac{\left(k,\,k\sqrt{q/a},\,-k\sqrt{q/a},\,a/k;\,q\right)_n}
{(\sqrt{aq},\,-\sqrt{aq},\,q k^2/a,\,q;\,q)_n}
\left(\frac{-k}{a}\right)^n, \notag
\end{align}
with $a=q$ and the appropriate choice $k$ (which we give below), in
\eqref{lambertrameq} provides representations of these theta
functions in terms of basic hypergeometric series.

For example, recall the function
\begin{equation}\label{aqeq}
a(q):=\sum_{m,n=-\infty}^{\infty}q^{m^2+mn+n^2}.
\end{equation}
This series was studied in \cite{BB91}, where it was shown that
\[
a^3(q)=b^3(q)+c^3(q),
\]where $b(q)=\sum_{m,n=-\infty}^{\infty}\omega^{m-n}q^{m^2+mn+n^2}$,
$\omega=exp(2\pi i/3)$, and
$c(q)=\sum_{m,n=-\infty}^{\infty}q^{(m+1/3)^2+(m+1/3)(n+1/3)+(n+1/3)^2}$.
The series $a(q)$ was also studied by Ramanujan, who showed
(\textbf{Entry 18.2.8} of Ramanujan's Lost Notebook - see \cite[page
402]{AB05}) that
\begin{equation}\label{corlameq2}
a(q)=1+6\sum_{n=1}^{\infty}\frac{q^{-2}q^{3n}}{1-q^{-2}q^{3n}}
-6\sum_{n=1}^{\infty}\frac{q^{-1}q^{3n}}{1-q^{-1}q^{3n}}.\notag
\end{equation}
In \cite[Equation (6.3), page 116]{B98} the author showed that
\begin{equation}\label{aq2eq}
a(q)-a(q^2)=6q\frac{(q,q^5,q^6,q^6;q^6)_{\infty}}{(q^2,q^4,q^3,q^3;q^6)_{\infty}}
=6q\frac{\psi^3(q^3)}{\psi(q)}.
\end{equation}

This result may be derived from  \eqref{Ffeq} by replacing $q$ with
$q^3$, $a$ with $1/q^3$ and $k$ with $1/q^2$, then inserting the
unit WP-Bailey pair \eqref{up} and finally using \eqref{corlameq2}
to combine the resulting Lambert series. Notice that the last
identity gives that $a(q)-a(q^2)=6q M(q,q^3)$, with the implications
given by \eqref{lambertrameq} noted above.

As a second example, consider the function $q \psi(q^2)\psi(q^6)$.
Ramanujan showed (see \textbf{Entry 3} (i), Chapter 19, page 223 of
\cite{B91}) that,
\begin{equation}\label{rqp2p6}
q\psi(q^2)\psi(q^6)=\sum_{n=1}^{\infty}\frac{q^{6n-5}}{1-q^{12n-10}}
-\sum_{n=1}^{\infty}\frac{q^{6n-1}}{1-q^{12n-2}}.
\end{equation}
With regard to \eqref{lambertrameq}, we note that $q
\psi(q^2)\psi(q^6)=q M(q,q^6)$.

Thirdly, recall the function
\begin{equation}\label{thetaeq}
\phi(q):=\sum_{n=-\infty}^{\infty}q^{n^2}=(-q,-q,q^2;q^2)_{\infty}.
\end{equation}
Ramanujan (\textbf{Entry 8 (i)} in chapter 17 of \cite{B91}) showed that
\[
\phi(q)^2=1+4 \sum_{n=1}^{\infty}\frac{q^{4n-3}}{1-q^{4n-3}}-4 \sum_{n=1}^{\infty}\frac{q^{4n-1}}{1-q^{4n-1}},
\]
and is easy to check that $\phi(q)^2=2 M(i,q)$, where $i^2=-1$.

\subsection{A variation of \eqref{lambertrameq}.}
It is not difficult to see that if we set $a = q^{2t+1}$, where $t$
is an integer, in \eqref{G'eq} and \eqref{Ffeq}, then the Lambert
series combine in essentially the same way as they did at
\eqref{lambeq1}, with a different finite sum of terms left over from
combining the Lambert series that essentially cancel each other,
(the particular sum depending on the choice of $t$). Corollary
\ref{cRS} follows from the choice $t=0$, but a similar result will
follow from other choices for $t$. If we let $t$ be a negative
integer, then inserting the usual WP-Bailey pairs in the resulting
identity will, in most cases, give trivial results (for most
WP-Bailey pairs, either the $\alpha_n$ or the $\beta_n$ contains a
$(a;q)_n$ factor, or some similar factor that will vanish for all
$n$ large enough when $a$ is a negative power of $q$). However,
there are two WP-Bailey pairs which give non-trivial results for the
choice $a=q^{-1}$, and we consider those results next. The proof of
the following corollary is virtually identical to that of Corollary
\ref{cRS2} (except, as mentioned, $a$ is set equal to $1/q$) and so
is omitted.

\begin{corollary}\label{cRS2}
Let $|q|<1$ and $(\alpha_n(a,k,q),\beta_n(a,k,q))$ be a WP-Bailey
pair. Define $\beta_n^{*}(1/q,k,q)=\lim_{a\to
1/q}(1-aq)\beta_n(a,k,q)$, if the limit exists. Then, assuming all
series converge, {\allowdisplaybreaks
\begin{multline}\label{lambertrameq2}
\sum_{n=1}^{\infty}\frac{(1-k q^{2n})(q;q)_{n-1}
\left(k^2q;q\right)_n(q^2;q^2)_{n-1}}
{(1-k)\left(\frac{1}{k},kq;q\right)_n \left(k^2q^2;q^2\right)_n}
\left(\frac{-1}{k}\right)^n\beta_n^{*}(1/q,k,q)
 \\
\phantom{asdfas} -\sum_{n=1}^{\infty} \frac{(q^2;q^2)_{n-1}
\left(k^2q;q^2\right)_n\alpha_{2n}(1/q,k,q)}
{\left(\frac{1}{k^2},q;q^2\right)_nk^{2n}}
\\
=k\frac{(k^2q,q/k^2,q^2,q^2;q^2)_{\infty}}{(k^2,q^2/k^2,q,q;q^2)_{\infty}}
\left(1+\sum_{n=0}^{\infty}\frac{ \left(k^2q^2,q;q^2\right)_{n}
\alpha_{2n+1}(1/q,k,q)}
{\left(\frac{q}{k^2},q^2;q^2\right)_{n}k^{2n}}\right).
\end{multline}
}
\end{corollary}

Inserting the WP-Bailey pair
\begin{align}\label{mz01}
\alpha_n^{(1)}(a,k)&=\frac{(q a^2/k^2;q)_n}{(q,q)_n}\left(
\frac{k}{a}\right)^n,\\
\beta_n^{(1)}(a,k)&=\frac{(q
a/k,k;q)_n}{(k^2/a,q,q)_n}\frac{(k^2/a;q)_{2n}}{(a q,q)_{2n}},\notag
\end{align}
leads, for $|k|>1$, to the identity
 {\allowdisplaybreaks
\begin{multline}\label{lambertrameq21}
\sum_{n=1}^{\infty}\frac{(1-k q^{2n}) \left(k^2q;q\right)_n}
{(1-q^n)(1-kq^n) \left(q;q^2\right)_n} \left(\frac{-1}{k}\right)^n
 -\sum_{n=1}^{\infty} \frac{
\left(k^2q,\frac{1}{k^2q};q^2\right)_nq^{2n}}
{(1-q^{2n})\left(q,q;q^2\right)_n}
\\
=k\frac{(k^2q,q/k^2,q^2,q^2;q^2)_{\infty}}{(k^2,q^2/k^2,q,q;q^2)_{\infty}}
\left(1+\sum_{n=0}^{\infty}\frac{\left(1-\frac{1}{k^2q} \right)
\left(k^2q^2,\frac{1}{k^2};q^2\right)_nq^{2n+1}}
{(1-q^{2n+1})\left(q^2,q^2;q^2\right)_n}\right).
\end{multline}
}

Similarly, inserting the pair at \eqref{pr4} gives the identity
 {\allowdisplaybreaks
\begin{multline}\label{lambertrameq22}
\sum_{n=1}^{\infty} \frac{(1-k q^{2n})
(q;q)_{n-1}\left(k^2q,k,\frac{1}{kq};q\right)_n q^n} {
\left(q,k^2q^2,\frac{1}{k},kq;q^2\right)_n}
\\
 -\sum_{n=1}^{\infty} \frac{(1-q^{4n-1})(q^2;q^2)_{n-1}
\left(k^2q,\frac{1}{k^2q^2},\frac{1}{q};q^2\right)_nq^{2n}}
{(1-q^{-1})\left(k^2q^3,\frac{1}{k^2},q,q^2;q^2\right)_n}
\\
=k\frac{(k^2q,q/k^2,q^2,q^2;q^2)_{\infty}}{(k^2,q^2/k^2,q,q;q^2)_{\infty}}.
\end{multline}
}


The next result also follows from \eqref{6psi6eq1a}, and expresses a
general sum involving an arbitrary WP-Bailey pair in terms of
Lambert series.

\begin{corollary}
Let $(\alpha_n(a,k),\beta_n(a,k))$ be a WP-Bailey pair. Then
{\allowdisplaybreaks\begin{multline}\label{GLambeq}
\sum_{n=1}^{\infty}\frac{(1-k q^{2n})(q,q;q)_{n-1}(qk^2;q^2)_n}
{(1-k)\left(kq,kq;q\right)_n\left(q;q^2\right)_n}
\left(-qk\right)^n\beta_n(k^2,k)\\
-
\sum_{n=1}^{\infty}\frac{(q^2,q^2;q^2)_{n-1}}
{\left(q^2k^2,q^2k^2;q^2\right)_n}
\left(qk\right)^{2n}\alpha_{2n}(k^2,k)\\
+\frac{\left(q^3k^2,q^3k^2,q^2,q^2;q^2\right)_{\infty}}
{\left(q^2k^2,q^2k^2,q,q;q^2\right)_{\infty}}\sum_{n=0}^{\infty}\frac{
\left(q,q;q^2\right)_n}
{\left(q^3k^2,q^3k^2;q^2\right)_n}
\left(qk\right)^{2n+1}\alpha_{2n+1}(k^2,k)\\
=\sum_{n=1}^{\infty}\frac{nq^n}{1-q^{2n}}
+\sum_{n=1}^{\infty}\frac{nk^{2n}q^{2n}}{1-q^{2n}}
-\sum_{n=1}^{\infty}\frac{nk^nq^n}{1-q^{n}}.
\end{multline}
}
\end{corollary}

\begin{proof}
Define
{\allowdisplaybreaks
\begin{multline}\label{6psi6eq1b}
G(k,q):=\sum_{n=1}^{\infty}\frac{(1-k q^{2n})(q,q;q)_{n-1}(qk^2;q^2)_n}
{(1-k)\left(kq,kq;q\right)_n\left(q;q^2\right)_n}
\left(-qk\right)^n\beta_n(k^2,k)\\
-
\sum_{n=1}^{\infty}\frac{(q^2,q^2;q^2)_{n-1}}
{\left(q^2k^2,q^2k^2;q^2\right)_n}
\left(qk\right)^{2n}\alpha_{2n}(k^2,k)\\
+\frac{\left(q^3k^2,q^3k^2,q^2,q^2;q^2\right)_{\infty}}
{\left(q^2k^2,q^2k^2,q,q;q^2\right)_{\infty}}\sum_{n=0}^{\infty}\frac{
\left(q,q;q^2\right)_n}
{\left(q^3k^2,q^3k^2;q^2\right)_n}
\left(qk\right)^{2n+1}\alpha_{2n+1}(k^2,k).
\end{multline}
}
From \eqref{6psi6eq1a}, it can be seen that $F(k^2,k,q)=0$ and that
\begin{equation}\label{FGeq}
G(k,q)=\lim_{a\to k^2}\frac{F(a,k,q)}{1-k^2/a} = \lim_{a\to k^2}\frac{a(F(a,k,q)-F(k^2,k,q))}{a-k^2}= k^2 M'(k^2),
\end{equation}
where $M(a):=F(a,k,q)$. The result follows from the representation of $M(a)$ as a  sum of Lambert series on the right side of \eqref{Ffeq}, after some simple algebraic manipulations, and after using the identity \[\sum_{n=1}^{\infty} \frac{x q^n}{(1-x q^n)^2}
 = \sum_{n=1}^{\infty}\frac{ n x^n q^n}{1- q^n}\]
a number of times.
\end{proof}
The identity at \eqref{GLambeq} extends a result by the first author in  a previous paper  \cite[Equation (7.3)]{McL09c}, where the identity which follows from \eqref{GLambeq} upon inserting the trivial WP-Bailey pair \eqref{tp} was proven.

Upon replacing $q$ with $q^2$ and $k$ with $1/q$ in \eqref{GLambeq}
and inserting the unit WP-Bailey pair \eqref{up}, we derive
Ramanujan's identity (\textbf{Example} (iii) on page 139 of
\cite{B91})
\begin{equation}\label{rampsieq}
q \psi^4(q^2)=\sum_{k=0}^{\infty}\frac{(2k+1)q^{2k+1}}{1-q^{4k+2}}.
\end{equation}
where $\psi(q)$ is defined at \eqref{psieq}. The same substitutions
in \eqref{GLambeq} followed by the insertion of the trivial
WP-Bailey pair \eqref{tp} gives the identity
(\cite[Corollary14]{McL09c})
\begin{equation}\label{waceq2}
\sum_{n=0}^{\infty}\frac{(1-q^{4n+3})(q^2;q^2)_{n}(q^4;q^4)_{n}(-q)^n}
{(1-q^{2n+1})(q^2;q^2)_{n+1}(q^2;q^4)_{n+1}}=\psi^4(q^2).
\end{equation}
It is possible to derive a more general identity involving the function $q \psi^4(q^2)$.

\begin{corollary}
Let $(\alpha_n(a,k,q),\beta_n(a,k,q))$ be a WP-Bailey pair. Then
{\allowdisplaybreaks
\begin{multline}\label{GLambeqmin1}
\sum_{n=1}^{\infty}\frac{1+ q^{2n}}{2}\frac{(q,q;q)_{n-1}q^n}
{\left(-q,-q;q\right)_n}
\beta_n(1,-1,q)\\
-\sum_{n=1}^{\infty}\frac{1+ q^{2n}}{2}\frac{(-q,-q;-q)_{n-1}\left(-q\right)^n}
{\left(q,q;-q\right)_n}
\beta_n(1,-1,-q)\\-\sum_{n=1}^{\infty}\frac{q^n}{(1-q^n)^2}\,\alpha_n(1,-1,q)
+\sum_{n=1}^{\infty}\frac{(-q)^n}{(1-(-q)^n)^2}\,\alpha_n(1,-1,-q)\\
=4q \psi^4(q^2).
\end{multline}
}
{\allowdisplaybreaks
\begin{multline}\label{GLambeqmin2}
\sum_{n=1}^{\infty}\frac{1+ q^{2n}}{1-q^{2n}}\frac{(q;q)_{n-1}q^{n(n+1)/2}}
{\left(-q;q\right)_n}
-\sum_{n=1}^{\infty}\frac{1+ q^{2n}}{1-q^{2n}}\frac{(-q;-q)_{n-1}(-q)^{n(n+1)/2}}
{\left(q;-q\right)_n}
+2q\times \\\sum_{n=1}^{\infty}q^{8 n^2-4 n}
\left(\frac{\left(q^{8 n-2}+3\right)
   q^{6 n-2}}{\left(1-q^{8 n-2}\right)^2}-\frac{\left(q^{4 n-2}+1\right) q^{-2
   n}}{\left(1-q^{4 n-2}\right)^2}+\frac{\left(3 q^{8
   n-6}+1\right) q^{2-6 n}}{\left(1-q^{8 n-6}\right)^2}\right) \\
=4q \psi^4(q^2).
\end{multline}
}
\end{corollary}

\begin{proof}
From \eqref{6psi6eq1b}, the left side of \eqref{GLambeqmin1} is $G(-1,q)-G(-1,-q)$. On the other hand, from \eqref{GLambeq},
\begin{align*}
G(-1,q)&=\sum_{n=1}^{\infty}\frac{nq^n}{1-q^{2n}}
+\sum_{n=1}^{\infty}\frac{nq^{2n}}{1-q^{2n}}
-\sum_{n=1}^{\infty}\frac{n(-1)^nq^n}{1-q^{n}}\\
&=\sum_{n=1}^{\infty}\frac{nq^n}{1-q^{n}}
-\sum_{n=1}^{\infty}\frac{n(-1)^nq^n}{1-q^{n}}\\
&=2\sum_{n=1}^{\infty}\frac{(2n-1)q^{2n-1}}{1-q^{2n-1}}\\
&=2\sum_{n=1}^{\infty}\frac{(2n-1)q^{2n-1}}{1-q^{4n-2}}
+2\sum_{n=1}^{\infty}\frac{(2n-1)q^{4n-2}}{1-q^{4n-2}}.
\end{align*}
Thus
\[
G(-1,q)-G(-1,-q)=4\sum_{n=1}^{\infty}\frac{(2n-1)q^{2n-1}}{1-q^{4n-2}},
\]
and \eqref{GLambeqmin1} follows from \eqref{rampsieq}.

For \eqref{GLambeqmin2}, we start with Singh's WP-Bailey pair \cite{S94},
{\allowdisplaybreaks
\begin{align*}
\alpha_{n}(a,k,q)&=\frac{(q \sqrt{a}, -q
\sqrt{a},a,\rho_1,\rho_2,a^2q/k\rho_1\rho_2;q)_n}
{(\sqrt{a},-\sqrt{a},q,a q/\rho_1,a q/\rho_2,k\rho_1\rho_2/a;q)_n}\left(\frac{k}{a}\right)^n,\\
\beta_n(a,k,q)&=\frac{(k \rho_1/a, k\rho_2/a, k,
aq/\rho_1\rho_2;q)_n}{(a q/\rho_1, a q/\rho_2, k \rho_1
\rho_2/a,q;q)_n}, \notag
\end{align*}}
set $a=1$,$k=-1$ and let $\rho_1, \rho_2 \to \infty$ to get the pair
{\allowdisplaybreaks
\begin{align*}
\alpha_{n}(1,-1,q)&=(1+q^n)(-1)^nq^{n(n-1)/2},\\
\beta_n(1,-1,q)&=\frac{(-1;q)_nq^{n(n-1)/2}}{(q;q)_n}. \notag
\end{align*}}
The first two series in \eqref{GLambeqmin2} come directly from inserting the expressions for $\beta_n(1,-1,q)$ and $\beta_n(1,-1,-q)$ in the first two series in
\eqref{GLambeqmin1}. The third series in \eqref{GLambeqmin2} comes inserting the expressions for $\alpha_n(1,-1,q)$ and $\alpha_n(1,-1,-q)$ in the last two series in  \eqref{GLambeqmin1}, then replacing  $n$, in turn, with $4n$, $4n-1$, $4n-2$ and $4n-3$, and then combining each pair of series into a single series, and finally combining the three surviving series together into one series.
\end{proof}

Remark: It is not difficult to see that a similar consideration of
\[
\lim_{a\to k^2}\frac{F(a,k,q)-F(1/a,1/k,k)}{1-k^2/a}
\]
at \eqref{6psi6eq1aa} gives the following result.
\begin{corollary}
Let $(\alpha_n(a,k),\beta_n(a,k))$ be a WP-Bailey pair and let $G(k,q)$ be as defined at
\eqref{6psi6eq1b}. Then
\begin{multline}\label{Geq2}
G(k,q)+G(1/k,q)\\
=\sum_{n=1}^{\infty}\frac{2nq^n}{1-q^{2n}}
+\sum_{n=1}^{\infty}\frac{nk^{2n}q^{2n}}{1-q^{2n}}
+\sum_{n=1}^{\infty}\frac{nq^{2n}/k^{2n}}{1-q^{2n}}
-\sum_{n=1}^{\infty}\frac{nk^nq^n}{1-q^{n}}
-\sum_{n=1}^{\infty}\frac{nq^n/k^n}{1-q^{n}}\\
=\frac{k(1-k^3)}{(1-k)(1-k^2)^2}
-k\frac{(q,q,-k^2,-q/k^2;q)_{\infty}(q^2,q^2;q^2)_{\infty}}
{(k^2,k^2,q^2/k^2,q^2/k^2;q^2)_{\infty}}\\
-k^2\frac{(k^2q,k^2q,q/k^2,q/k^2,q^2,q^2,q^2,q^2;q^2)_{\infty}}
{(k^2,k^2,q^2/k^2,q^2/k^2,q,q,q,q;q^2)_{\infty}}.
\end{multline}
\end{corollary}

The special case of the this identity that follows from inserting the trivial pair \eqref{tp} into the term $G(k,q)+G(1/k,q)$ also follows from Corollary 12 in  \cite{McL09c}, upon dividing the identity there by $1-b$ and then letting $b \to 1$.

As well as implying some of the known identities relating Lambert
series and infinite products, the term $G(k,q)+G(1/k,q)$ in
\eqref{Geq2} also provides an additional expression for the Lambert
series or the infinite product in terms of series involving an
arbitrary WP-Bailey pair $(\alpha_n(a,k),\beta_n(a,k))$ (with
$a=k^2$ and $k$ specialized as required). We give two examples.

First recall that
\[
\phi(q):=\sum_{n=-\infty}^{\infty}q^{n^2}=(-q;q^2)_{\infty}^2(q^2;q^2)_{\infty}
=\frac{(-q,q^2;q^2)_{\infty}}{(q,-q^2;q^2)_{\infty}}.
\]
 \begin{corollary}
 If $|q|<1$, then
 {\allowdisplaybreaks\begin{align}\label{phi4eq}
 1+8\sum_{n=1}^{\infty}
 \frac{nq^n}{1+(-q)^n}&=\phi^4(q)\\
 &=1+4\sum_{n=1}^{\infty}
 \frac{(1-iq^{2n})(q,q;q)_{n-1}(-1;q)_{2n}(-iq)^n}
 {(1-i)(iq,iq;q)_n(q;q)_{2n}}\notag\\
 &\phantom{=1}+4\sum_{n=1}^{\infty}
 \frac{(1+iq^{2n})(q,q;q)_{n-1}(-1;q)_{2n}(iq)^n}
 {(1+i)(-iq,-iq;q)_n(q;q)_{2n}}.\notag
 \end{align}}
 \end{corollary}
\begin{proof}
Let $k=i$ in \eqref{Geq2}. It is not difficult to show that the Lambert series combine to give
\[
2 \sum_{n=1}^{\infty}\frac{2nq^{2n}}{1+q^{2n}}
+2 \sum_{n=1}^{\infty}\frac{(2n-1)q^{2n-1}}{1-q^{2n-1}}
=2\sum_{n=1}^{\infty}
 \frac{nq^n}{1+(-q)^n}.
\]
It is also easy to see that the infinite product side simplifies to give
\[
-\frac{1}{4} +\frac{1}{4}\phi^{4}(q).
\]
For $G(i,q)+G(1/i,q)$, we use \eqref{6psi6eq1b} with $k=i$ and insert the trivial WP-Bailey pair \eqref{tp} (with $a=k^2$ and then $k=i$):
\begin{align*}
\alpha_{n}&=
0, \,n>0, \\
\beta_n&=\frac{(i,-i;q)_n}{(-q,q;q)_n}=
\frac{(-1;q^2)_n}{(q^2;q^2)_n}.\notag
\end{align*}
Multiply the resulting set of equalities by 4, add 1, and the identities at \eqref{phi4eq} follow.
\end{proof}
Remarks: The first equality at \eqref{phi4eq} is due to Jacobi  \cite{J29}.  Also, if other WP-Bailey pairs are used instead of the trivial pair above, then still further representations for  $\phi^{4}(q)$ will result.

\begin{corollary}\label{c9}
Let $\omega:=\exp(2 \pi i/3)$,   let  $\chi_0(n)$ denote the
principal character modulo 3, and let $\psi(q)$ be as defined at
\eqref{psieq}. Then {\allowdisplaybreaks\begin{align}\label{chi3eq}
9\sum_{n=1}^{\infty}\chi_0(n)\frac{nq^n}{1-q^{2n}}&=\frac{\psi^6(q)}{\psi^2(q^3)}
-\frac{\psi^3(q^{1/2})\psi^3(-q^{1/2})}{\psi(q^{3/2})\psi(-q^{3/2})}\\
&=3(1-\omega^2)\sum_{n=1}^{\infty}\frac{(1-\omega
q^{2n})(q;q)_{n-1}(\omega^2q;q^2)_n(-\omega q)^n}
{(1-q^{3n})(\omega q;q)_n(q;q^2)_n}\notag\\
&\phantom{asd}+3(1-\omega)\sum_{n=1}^{\infty}\frac{(1-\omega^2
q^{2n})(q;q)_{n-1}(\omega q;q^2)_n(-\omega^2 q)^n}
{(1-q^{3n})(\omega^2 q;q)_n(q;q^2)_n}.\notag
\end{align}}
\end{corollary}
\begin{proof}
The proof is similar to that of the previous corollary, except we set $k=\omega$ in \eqref{Geq2}. The Lambert series combine to give
\[
3 \sum_{n=1}^{\infty}\frac{(3n-1)q^{3n-1}}{1-q^{6n-2}}
+3 \sum_{n=1}^{\infty}\frac{(3n-2)q^{3n-2}}{1-q^{6n-4}}
=3\sum_{n=1}^{\infty}\chi_0(n)\frac{nq^n}{1-q^{2n}}.
\]
The infinite product side simplifies to give
\[
\frac{1}{3}\frac{\psi^6(q)}{\psi^2(q^3)}
-\frac{1}{3}\frac{\psi^3(q^{1/2})\psi^3(-q^{1/2})}{\psi(q^{3/2})\psi(-q^{3/2})}.
\]
Once again, for $G(\omega,q)+G(1/\omega,q)$ we use \eqref{6psi6eq1b}, this time with $k=\omega$ and  the trivial WP-Bailey pair \eqref{tp} (with $a=k^2$ and then $k=\omega$):
\begin{align*}
\alpha_{n}&=
0, \,n>0, \\
\beta_n&=\frac{(\omega,\omega^2;q)_n}{(\omega^2q,q;q)_n}.\notag
\end{align*}
Multiply the resulting set of equalities by 3, and the identities at \eqref{chi3eq} follow, after some simple manipulations of the expressions for $G(\omega,q)$ and $G(\omega^2,q)$.
\end{proof}

Remark: The Lambert series in the identity at \eqref{chi3eq} may also
be represented in terms of the theta series $a(q)$   defined at\eqref{aqeq}, upon noting that
\[
\sum_{n=1}^{\infty}\chi_0(n)\frac{nq^n}{1-q^{2n}}
=\sum_{n=1}^{\infty}\chi_0(n)\frac{nq^n}{1-q^{n}}
-\sum_{n=1}^{\infty}\chi_0(n)\frac{nq^{2n}}{1-q^{2n}},
\]
and employing an identity of Ramanujan from the Lost Notebook, \textbf{Entry 18.2.9} (see \cite{AB05}, page 402), which states that
\begin{equation*}
a^2(q)=1+12\sum_{n=1}^{\infty}\chi_0(n)\frac{nq^n}{1-q^{n}}.
\end{equation*}

 \allowdisplaybreaks{

}
\end{document}